\documentclass{amsart}
\usepackage{amsmath, amscd, amssymb, amsthm}
\usepackage{bbm}
\usepackage{latexsym}
\usepackage{amsfonts}
\usepackage{graphicx}
\usepackage[all,cmtip]{xy}
\usepackage[colorlinks,linkcolor=blue,breaklinks=blue,urlcolor=blue,citecolor=blue,anchorcolor=blue,pagebackref]{hyperref}%
\usepackage{geometry}
\newtheorem{theorem}{Theorem}
\newtheorem{lemma}{Lemma}
\newtheorem{corollary}[theorem]{Corollary}
\newtheorem{remark}{Remark}

\newtheorem*{definition}{Definition}
\newtheorem{conjecture}{Conjecture}

\renewcommand*\backref[1]{}
\renewcommand*\backrefalt[4]{ \ifcase #1 \or (cited on page #2) \else (cited on pages #2) \fi}

\newcommand{\be}{\begin{equation}}
\newcommand{\ee}{\end{equation}}
\newcommand{\bea}{\begin{eqnarray}}
\newcommand{\eea}{\end{eqnarray}}

\newcommand{\vs}{\vspace{0.5cm}}

\newcommand{\vsv}{\vspace{0.12cm}}

\def\XXint#1#2#3{{\setbox0=\hbox{$#1{#2#3}{\int}$ }
\vcenter{\hbox{$#2#3$ }}\kern-.6\wd0}}

\begin{document}

\title[On Strominger K\"ahler-like manifolds]{On Strominger K\"ahler-like manifolds with degenerate torsion }

\author{Shing-Tung Yau} \thanks{The research of Yau is partially supported by NSG grants PHY-0714648 and DMS-1308244. Zhao is partially supported by NSFC with Grant No.11801205 and China Scholarship Council to Ohio State University. Zheng is partially supported by NSFC grant 12071050 and 12141101, a Chongqing grant cstc2021ycjh-bgzxm0139, and the 111 Project D21024.}

\address{Shing-Tung Yau. Department of Mathematics, Harvard University, Cambridge, MA 02138, USA}
\email{{yau@math.harvard.edu}}

\author{Quanting Zhao} \thanks{}
\address{Quanting Zhao. School of Mathematics and Statistics \& Hubei Key Laboratory of Mathematical Sciences, Central
China Normal University, Wuhan 430079, China.}

\email{{zhaoquanting@126.com; zhaoquanting@mail.ccnu.edu.cn}}

\author{Fangyang Zheng} \thanks{}
\address{Fangyang Zheng. School of Mathematical Sciences, Chongqing Normal University, Chongqing 401331, China}
\email{{20190045@cqnu.edu.cn; franciszheng@yahoo.com}}


\begin{abstract}
In this paper, we study a special type of compact Hermitian manifolds that are Strominger K\"ahler-like, or SKL for short. This condition means that the Strominger connection (also known as Bismut connection) is K\"ahler-like, in the sense that its curvature tensor obeys all the symmetries of the curvature of a K\"ahler manifold. Previously, we have shown that any SKL manifold $(M^n,g)$ is always pluriclosed, and when the manifold is compact and $g$ is not K\"ahler, it can not admit any balanced or strongly Gauduchon (in the sense of Popovici) metric. Also, when $n=2$, the SKL condition is equivalent to the Vaisman condition. In this paper, we give a classification for compact non-K\"ahler SKL manifolds in dimension $3$ and those with degenerate torsion in higher dimensions. We also present some properties about SKL manifolds in general dimensions, for instance, given any compact non-K\"ahler SKL manifold, its K\"ahler form represents a non-trivial Aeppli cohomology class, the metric can never be locally conformal K\"ahler when $n\geq 3$, and the manifold does not admit any Hermitian symplectic metric.
\end{abstract}

\maketitle

\tableofcontents

\markleft{Yau, Zhao and Zheng}
\markright{Strominger K\"ahler-like}

\section{Introduction and statement of results}

For a Hermitian manifold $(M^n,g)$, its {\em Strominger connection} $\nabla^s$ is the unique connection on $M$ that is Hermitian (namely, $\nabla^sg=0$, $\nabla^s J =0$) and has totally skew-symmetric torsion tensor. Its existence and explicit expression first appeared in Strominger's seminal paper \cite{Strominger} in 1986, where he called it the H-connection. Three years later, Bismut \cite{Bismut} discovered the connection independently and used it in his study of local index theorems, which leads to the name {\em Bismut connection} in many literature. Since Strominger's paper was published earlier than Bismut's, it might be more appropriate to call it Strominger connection, and we shall do so from now on. Note that the connection also appeared implicitly earlier (see \cite{Yano}) and in some literature it was also called the {\em KT connection} (K\"ahler with torsion) or {\em characteristic connection}. 

Since the need of non-K\"ahler Calabi-Yau spaces in string theory, this connection has been receiving more and more attention from geometers and mathematical physicists alike. We refer the readers to \cite{AI}, \cite{AU}, \cite{EFV}, \cite{FV}, \cite{Fu},  \cite{Fu-Li-Yau}, \cite{Fu-Wang-Wu}, \cite{Fu-Yau}, \cite{Fu-Zhou},  \cite{GatesHR}, \cite{Gauduchon1}, \cite{IvanovP}, \cite{KYZ}, \cite{Li-Yau}, \cite{Liu-Yang}, \cite{Liu-Yang1}, \cite{Liu-Yang2},   \cite{S18}, \cite{STW}, \cite{Tosatti}, \cite{Tseng-Yau}, \cite{VYZ}, \cite{YZ1}, \cite{Zheng1} and the references therein for more discussions on Strominger connection, pluriclosed metric and related topics.

Throughout this paper, we will call a Hermitian manifold $(M^n,g)$ whose Strominger connection is K\"ahler-like\footnote{The definition of a metric connection on a Hermitian manifold being K\"ahler-like is given by Angella, Otal, Ugarte and Villacampa in \cite{AOUV}. For the special case of Riemannian and Chern connections, it was studied by Bo Yang and the third named author in \cite{YZ}. The concept originated from the earlier works of Alfred Gray \cite{Gray} and others in 1960s.} a {\em Strominger K\"ahler-like} manifold, or a {\em SKL} manifold in short. The structure equations, Bianchi identities and notations alike used in \cite[Section 2]{ZZ} will also be applied here to investigate the SKL geometry. In our previous work \cite{ZZ}, we have shown that a Hermitian manifold $(M^n,g)$ is SKL if and only if the metric is pluriclosed, namely, $\partial \overline{\partial}\omega_g=0$ where $\omega_g$ is the K\"ahler form of $g$, and the torsion is $\nabla^s$-parallel. Pluriclosed metrics (also known as {\em strong K\"ahler with torsion,} or SKT metric) are widely studied in recent years, and we refer the readers to the excellent survey paper by Fino and Tomassini \cite{FinoTomassini} for more information on this type of special Hermitian metrics.

It has been proved in \cite[Theorem 3 and 4]{ZZ} that, if $(M^n,g)$ is a compact SKL manifold with $g$ not K\"ahler, then $M^n$ cannot admit any balanced metric, or more generally, it can not admit any strongly Gauduchon metric (in the sense of Popovici \cite{Popovici}). Furthermore, it has been shown in \cite[Theorem 2]{ZZ} that, when $n=2$, the SKL condition is equivalent to the Vaisman condition, which means that the Lee form is parallel under the Riemannian (Levi-Civita) connection. Compact Vaisman surfaces were fully classified by the beautiful work of Belgun \cite{Belgun} and they are non-K\"ahler properly elliptic surfaces, Kodaira surfaces, and Class 1 or elliptic Hopf surfaces \cite{Kodaira}.

The first result of this paper is the following observation. Recall that a Hermitian metric $\omega$ is called {\em Hermitian symplectic,} if there exists a $(2,0)$-form $\alpha$ on $M^n$ such that $\partial \omega = -\overline{\partial} \alpha$ and $\partial \alpha =0$. Equivalently, there exists a $(2,0)$-form $\alpha$ on the manifold such that $d(\alpha + \omega + \overline{\alpha })=0$. Such a metric is always pluriclosed, namely, $\partial\overline{\partial}  \omega=0$.

\begin{theorem} \label{thm1}
Let $(M^n,g)$ be a compact SKL manifold with $g$ not K\"ahler. Then $\omega_g$ represents a non-trivial Aeppli cohomology class in $H^{1,1}_A(M)$. Furthermore, $M^n$ does not admit any Hermitian symplectic metric. In particular, $g$ is a pluriclosed but not Hermitian symplectic metric and $M^n$ does not satisfy the $\partial\overline{\partial}$-Lemma.
\end{theorem}

Here $H^{p,q}_A(M)$ stands for the $(p,q)$-Aeppli cohomology group, which is defined by
$$ H^{p,q}_A(M) = \frac{\mbox{ker}( \partial \overline{\partial}: A^{p,q} \rightarrow A^{p\!+\!1,q\!+\!1})} {\partial A^{p\!-\!1,q} + \overline{\partial} A^{p,q\!-\!1}}$$
where $A^{p,q}$ is the space of all complex valued $(p,q)$-forms on $M^n$. Either by the fact that a compact non-K\"ahler SKL manifold does not satisfy the $\partial \overline{\partial}$-Lemma, or by the non-existence of any balanced metric on such a manifold, we conclude that

\begin{remark}
A compact complex manifold in the Fujiki class (namely it is bimeromorphic to a compact K\"ahler manifold) does not
admit any non-K\"ahler SKL metric. In particular, the Kodaira dimension $\mbox{kod}(M^n)$ of a compact non-K\"ahler SKL manifold $(M^n,g)$ can never be equal to $n$.
\end{remark}

Note that in dimension $2$, compact non-K\"ahler SKL surfaces are precisely the compact Vaisman surfaces with odd $b_1$, which are classified by Belgun in \cite{Belgun}. Their Kodaira dimensions can already be $1$, $0$, or $-\infty$. Another general property about SKL manifolds is

\begin{theorem} \label{thm2}
Let $(M^n,g)$ be a SKL manifold with $g$ not K\"ahler. Then there exists a holomorphic vector field $X$ on $M$ which is parallel with respect to the Strominger connection $\nabla^s$ of $g$. In particular, the norm $|X|$ is a positive constant and the Euler number of $M$ is zero.
\end{theorem}

Our next observation is about the uniqueness of SKL metrics within a conformal class. Note that since SKL metrics are Gauduchon by \cite[Proposition 3]{ZZ}, so when $M^n$ is compact, any SKL metric on $M^n$ will be unique (up to constant multiple) within its conformal class. The same is true for Riemannian K\"ahler-like or Chern K\"ahler-like metrics as proved in \cite[Theorem 4]{YZ}.  When $M^n$ is not compact, however, Riemannian K\"ahler-like or Chern K\"ahler-like metrics are no longer unique within a conformal class, but SKL metrics are, provided that the dimension is at least $3$:

\begin{theorem} \label{thm3}
Let $(M^n,g)$ be any Hermitian manifold with $n\geq 3$. Then within the conformal class of $g$, there is at most one SKL metric, up to constant multiples.
\end{theorem}

As mentioned above, in the case of $n=2$, a SKL metric is actually Vaisman, namely a Hermitian metric which is locally conformal K\"ahler with its  Lee form parallel under the Levi-Civita connection. Hence, on the universal cover, any SKL metric $g$ on $M^2$ is conformal to a K\"ahler metric, and thus is not unique within its conformal class when $g$ is not K\"ahler.
When $n\geq 3$, however, Theorem \ref{thm3} implies that any non-K\"ahler SKL metric is never locally conformal K\"ahler. We speculate that there cannot exist any other locally conformal K\"ahler metrics as well:

\begin{conjecture}\label{no_lcK}
If $(M^n,g)$ is a compact SKL manifold with $g$ not K\"ahler and $n\geq 3$, then $M^n$ does not admit any locally conformal K\"ahler metric.
\end{conjecture}

As a partial evidence, we prove the following:

\begin{theorem} \label{thm4}
Let $(M^n,g)$ be a compact SKL manifold with $g$ not K\"ahler. If $n\geq 3$, then $M^n$ cannot admit any Vaisman metric.
\end{theorem}

Note that a compact Hermitian manifold $(M^n,g)$ is called \emph{Calabi-Yau with torsion} or CYT in short, if its Strominger connection $\nabla^s$ has holonomy in $SU(n)$, that is, the first Ricci curvature of $\nabla^s$ is identically zero. If $g$ is K\"ahler, then it is a compact Ricci flat K\"ahler manifold, often called a Calabi-Yau space (in the broader sense). Assume that $g$ is not K\"ahler. It was proved in \cite{IvanovP} that if $(M^n,g)$ is CYT and $g$ is pluriclosed (and non-K\"ahler), then the plurigenera are all zero, namely, the Kodaira dimension of  $M$ is $-\infty $. Since all SKL manifolds are pluriclosed, we get as a consequence that

\begin{remark} \label{thm5}
If $(M^n,g)$ is a compact non-K\"ahler SKL manifold that is CYT, (or more generally if the total scalar curvature of the Strominger connection is nonnegative), then its Kodaira dimension $\mbox{kod}(M^n) = -\infty $.
\end{remark}

In fact, it seems to us that SKL and CYT together would make a very restrictive situation, and we would like to propose the following:

\begin{conjecture}
Let $(M^n,g)$ be a compact SKL manifold with $n\geq 2$. Assume that the universal cover of $M^n$ does not admit any K\"ahler de Rham factor of dimension bigger than $1$. If the Strominger connection $\nabla^s$ has vanishing first Ricci curvature, then $g$ is Strominger flat.
\end{conjecture}

In other words, we conjecture that compact non-K\"ahler SKL manifolds (without K\"ahler de Rham factors of dimension bigger than $1$) that are CYT must be Strominger flat. As a supporting evidence, we show that it is true in dimension $2$ or $3$:

\begin{theorem} \label{thm6}
Let $(M^n,g)$ be a non-K\"ahler SKL manifold. If $n\leq 3$ and the Strominger connection has the vanishing first Ricci curvature, then it is Strominger flat.
\end{theorem}

Note that Strominger flat manifolds were classified in \cite{WYZ}, they are quotients of Samelson spaces \cite{S}, namely Lie groups equipped with bi-invariant metrics and compatible left invariant complex structures. This is analogous to the classic result of Boothby \cite{Boothby} which states that any compact Chern flat manifold is a quotient of a complex Lie group equipped with a left-invariant metric (see also \cite{Wang}).

The above results are pretty much all on the negative side, illustrating how restrictive the class of SKL manifolds is. On the existence side, for $n=2$, since SKL is equivalent to Vaisman, we know from the work of Belgun \cite{Belgun} that there are three types of compact non-K\"ahler SKL surfaces: the (non-K\"ahler) properly elliptic surfaces, the Kodaira surfaces, and some (but not all) Hopf surfaces, whose Kodaira dimensions are $1$, $0$, and $-\infty$, respectively. All SKL complex nilmanifolds with nilpotent complex structure were classified in \cite{ZZ1}. They turned out to be a very special type of step (at most) two nilpotent Lie groups and explicit descriptions were given there.

As noted in \cite{WYZ}, a central Calabi-Eckmann threefold $S^3\times S^3$ is Strominger flat, hence SKL. More generally, if $N_1$ and $N_2$ are two Sasakian $3$-manifolds, then the natural Hermitian structure on the product manifold $M^3=N_1\times N_2$ is necessarily SKL.

\begin{definition} \label{Sasakian}
Recall that a Sasakian manifold $(N^{2m+1}, g, \xi)$ is an odd dimensional Riemannian manifold $(N,g)$ equipped with a Killing vector field $\xi$ with unit norm, such that:
\begin{enumerate}
\item The tensor field $\frac{1}{c}\nabla \xi$ (where $c>0$ is a constant), which sends a tangent vector $X$ to the tangent vector $\frac{1}{c}\nabla_X\xi$, gives an integrable orthogonal complex structure $J$ on the distribution $H$, where $H$ is the perpendicular complement of $\xi $ in the tangent bundle $TN$.
\item Denote by $\alpha $ the $1$-form dual to $\xi$, namely, $\alpha (X) = g(X, \xi)$ for any $X$, then $\alpha \wedge (d\alpha )^m$ is nowhere zero. That is, $\alpha$ gives a contact structure on $N$.
\end{enumerate}
\end{definition}

\begin{definition} \label{HSasakian}
Let $(N_1^{2n_1+1}, g_1, \xi_1)$ and $(N_2^{2n_2+1}, g_2, \xi_2)$ be two Sasakian manifolds. On the product Riemannian manifold $M=N_1\times N_2$, of even dimension $2n=2(n_1+n_2+1)$, consider the natural almost complex structure $J$ defined by (where $c_i>0$ are constants)
$$ J\xi_1 = \xi_2, \ \ JX_i = \frac{1}{c_i}\nabla_{X_i}\xi_i \ \ \forall \ X_i \in H_i , \ i=1,2 $$
It is well-known to be integrable. We will call the Hermitian manifold  $(N_1\times N_2, g_1\times g_2, J)$ the standard Hermitian structure on the product of two Sasakian manifolds.
\end{definition}

Note that our notations for $J$ here is slightly more general in the sense that we allow the two scaling constants $c_1$ and $c_2$ here in the construction of the complex structure, namely, the K\"ahler form of the metric $g$ is given by
$$ \omega = \frac{1}{2c_1} d\alpha_1 + \frac{1}{2c_2} d\alpha_2 + \alpha_1\wedge \alpha_2. $$
It is proved by Belgun \cite[Proposition 3.2]{Belgun2} that for the Hermitian manifold $(M^n,g)=(N_1\times N_2, g_1\times g_2)$, its Strominger connection $\nabla^s$ always has parallel torsion (this is also true when the two positive constants $c_1$,  $c_2$ are not $1$). On the other hand, it is well known to experts that the metric $g$ will be pluriclosed if and only if both $n_1\leq 1$ and $n_2\leq 1$, since
$$ \sqrt{\!-\!1}\partial \overline{\partial} \omega =  d\alpha_1\wedge d\alpha_1 + d\alpha_2 \wedge d\alpha_2. $$
See for instance \cite[Formula (4.3)]{Matsuo}, where his $\Omega$ stands for the K\"ahler form $\omega$ and  the coefficients are $a=0$, $b=1$, $ \Phi_i = d \alpha_i$, with $\alpha_i$ being the contact form as in the above definition for $i=1,2$. The above formula implies that $g$ is pluriclosed when both $n_1\leq 1$ and $n_2\leq 1$,  and $g$ is not pluriclosed if either $n_1$ or $n_2$ is bigger than $1$.

Since SKL means pluriclosed plus $\nabla^s$ has parallel torsion by \cite[Theorem 1]{ZZ}, we know that the product of two Sasakian manifolds will be SKL if and only if both factors are of real dimension $3$ or $1$:

\begin{corollary} \label{thm7}
Let $M^n$  be the standard Hermitian manifold on the product of two Sasakian manifolds, of complex dimension $n=n_1+n_2+1\geq 2$. Then $M^n$ is SKL if and only if both $n_1\leq 1$ and $n_2\leq 1$.
\end{corollary}

Here we ignored the trivial case of $n_1=n_2=0$. When $n\geq 2$, $M$ is always non K\"ahler, and the condition $n_1, n_2\leq 1$ means  either $n=2$ and $M$ is the product of a Sasakian $3$-manifold with the circle $S^1$ (or ${\mathbb R}$), or $n=3$ and $M$ is the product of two Sasakian $3$-manifolds.

Let $N^3$ be a complete, simply-connected Sasakian $3$-manifold. When $N$ is compact, or more generally when $N$ is co-compact in the sense that there is a compact subset $K\subseteq N$ and a group $\Gamma$ of isometries of $N$ preserving the Sasakian structure such that the union of $h(K)$ for all $h\in \Gamma$ covers $N$, then $N$ is classified by Belgun \cite{Belgun}, \cite{Belgun1}, \cite{Belgun2}. In particular, it was shown in \cite[Theorem 4.5]{Belgun2}  that after the so-called parallel modification, $N^3$ can be deformed to one of three standard Lie groups with  left invariant Sasakian structures: $SU(2)$, $\widetilde{SL}(2, {\mathbb R})$, and $Nil_3$.

The main purpose of this paper to give the following classification theorem for three dimensional non-K\"ahler SKL manifolds, which says that all such manifolds are given by Corollary \ref{thm7}:

\begin{theorem}\label{thm8}
Let $M^n$ be a complete, non-K\"ahler SKL manifold. Let $\widetilde{M}$ be its universal cover.
\begin{enumerate}
\item If $n=2$, then $\widetilde{M} = N^3\times {\mathbb R}$ is the product of a Sasakian $3$-manifold $N^3$ with ${\mathbb R}$.
\item If $n=3$, then either $\widetilde{M}$ is holomorphically isometric to $M_1^2 \times C$, where $M^2_1$ is a non-K\"ahler SKL surface and $C$ is a K\"ahler curve, or $\widetilde{M}=N_1^3 \times N_2^3$ is the product of two Sasakian $3$-manifolds.
    \end{enumerate}
\end{theorem}

When $M^2$ is compact, Belgun's work \cite{Belgun} says that $\widetilde{M}$ is biholomorphic to either ${\mathbb C}^2\setminus \{ 0\}$, or ${\mathbb C}^2$, or ${\mathbb C}\times D$ (with $D$ the unit disc), while $M^2$ is a Hopf surface, a Kodaira surface, or a non-K\"ahler properly elliptic surface (which after a finite cover is a holomorphic fiber bundle over a curve of genus at least $2$ with fiber being a smooth elliptic curve). Similarly, when $M^3$ is compact, the factor $N^3$ for $M_1^2$ in the first case or the factors $N_1$, $N_2$ in the second case, are all co-compact in the sense of Belgun \cite{Belgun2}, hence will be one of the three types mentioned above.

It is natural to wonder in Theorem \ref{thm8} what can one say about the deck transformation group when $M^n$ is compact. In particular, when $n=3$ and $\widetilde{M}$ is the product of two Sasakian $3$-manifolds, would $M^3$ itself (or a finite cover of it) be the product of a two Sasakian $3$-manifolds? We believe that this should be the case, but we are unable to prove it at this point. In the special case when a compact $M^n$ is Strominger flat, the behavior of deck transformation group was discussed in \cite{WYZ}. 

Next, motivated by the notion of {\em locally conformal K\"ahler metric with potential} by Ornea and Verbitsky \cite{OV}, Belgun introduced in \cite{Belgun2} the notion of {\em Lee potential} (LP in short), and the notion of {\em Generalized Calabi-Eckmann} (GCE in short) for Hermitian manifolds, where he gave a full classification of all compact GCE threefolds in \cite[Theorem 4.5]{Belgun2}.

\begin{definition}[\cite{Belgun2}]
A Hermitian manifold $M^n$ is LP if the Gauduchon torsion $1$-form $\eta$ satisfies
\[\eta \neq 0, \ \ \ \ \partial \eta =0 , \ \ \ \ \partial \omega = c \,\eta \,\partial \overline{\eta}\]
where $c$ is a non-zero constant. A Hermitian manifold is GCE if it is LP and $\nabla^sT^s=0$, namely, the torsion of the Strominger connection $\nabla^s$ is parallel with respect to $\nabla^s$.
\end{definition}

We observe that when $n=3$, the condition $\nabla^sT^s=0$ actually implies the LP condition if the Hermitian metric is not balanced, hence GCE simply means $\nabla^sT^s=0$ and non-balanced in dimension $3$:

\begin{theorem} \label{thm9}
Let $M^3$ be a Hermitian manifold which is not balanced and the torsion of its Strominger connection $\nabla^s$ is parallel with respect to $\nabla^s$. Then it satisfies the LP condition in the sense of Belgun and thus $M^3$ is GCE. In particular, any non-K\"ahler SKL threefold is GCE.
\end{theorem}

The converse of the above is not true in general, as the SKL is equivalent to the parallelness of the torsion of $\nabla^s$ plus the pluriclosedness $\partial \overline{\partial }\omega =0$, which in this case is equivalent to $\overline{\partial }\eta \wedge \partial \overline{ \eta} =0$. Therefore, for $n=3$, the SKL condition is more restrictive than GCE, while for $n\geq 4$, they don't have much in common, as we shall see below.

Note that the LP condition basically says that the torsion tensor contains the same amount of information as the torsion $1$-form $\eta$, and for $n\geq 4$, the SKL manifolds do not satisfy the LP condition in general. We will introduce the concept of {\em degenerate torsion} for non-K\"ahler SKL manifolds in Section \ref{degtor}. It turns out that a SKL manifold of the dimension $2$ or $3$, or a SKL manifold that is LP, will always have degenerate torsion as shown in Lemma \ref{lemma8} and the following theorem. For such manifolds, we have the splitting result.

\begin{theorem} \label{thm10}
For a non-K\"ahler SKL manifold, the LP condition is equivalent to the degenerate torsion condition. Furthermore, if $M^n$ is a complete non-K\"ahler SKL manifold with degenerate torsion, then its universal cover is holomorphically isometric to a product $M_1^k \times M_2^{n-k}$, where $M_2$ is K\"ahler, and $M_1$ has complex dimension $k=2$ or $3$.
\end{theorem}

This result illustrates the point that on one hand, there seems to be a distinction between dimension $n\leq 3$ and dimensions $n\geq 4$ for SKL manifolds, as the torsion tensor is degenerate in the first case while non-degenerate in general in the second case. So the study of SKL manifolds in dimensions $n\geq 4$ might be considerably more complicated. On the other hand, the classification theorem for SKL complex nilmanifolds \cite{ZZ1} seems to suggest that, at least in some special cases, one could still expect SKL manifolds to obey some very restrictive pattern.

Note that complex nilmanifolds have trivial canonical line bundle, as it is easy to verify that $\varphi_1 \wedge \varphi_2 \wedge \cdots \wedge \varphi_n$ is $d$-closed by Salamon's Theorem \cite[Theorem 1.3]{Sal}, where $\varphi$ is a unitary left invariant coframe. Therefore, SKL complex nilmanifolds can be considered as high dimensional generalization of Kodaira surfaces. We wonder if they are basically the only compact non-K\"ahler SKL manifolds with trivial canonical line bundle, up to deformation of complex structures and SKL metrics, which motivates the following conjecture

\begin{conjecture}
Let $(M^n,g)$ be a compact SKL manifold with $g$ non-K\"ahler. If the canonical line bundle is trivial,  then $(M^n,g)$ can be deformed to a complex nilmanifold $(N^n,h)$, namely, $N=G/\Gamma$ where $G$ is a nilpotent Lie group and $\Gamma$ a cocompact lattice, and $h$ is a left invariant metric compatible with a left invariant complex structure on $G$. In this case, the step of $G$ is at most two, the left invariant complex structure on $G$ is necessarily abelian and its structure is given explicitly by \cite[Theorem 1]{ZZ1}.
\end{conjecture}

Here by deformation we mean a smooth path $(M_t, g_t)$ of compact SKL manifolds with trivial canonical line bundle,  which starts with $(M,g)$ at $t=0$ and ends with $(N,h)$ at $t=1$. Belgun's work \cite{Belgun} and \cite{Belgun2} says that, in the $n=2$  and $n=3$ cases, any SKL surface or threefold can be deformed to  homogenous ones, which have constant Strominger scalar curvature $S$.  For $n=2$, he also showed that any SKL metric can be deformed to one where the Lee form has unit length.

\section{Properties of SKL manifolds}\label{plr}

Given a Hermitian manifold $(M^n,g)$, denote by $\omega$ its K\"ahler form. There are several well studied generalizations to the K\"ahlerness condition $d\omega =0$:
\begin{enumerate}
\item $g$ is {\em balanced,} if $d(\omega^{n-1})=0$ (that is, the Gauduchon's {\em torsion $1$-form} $\eta =0$).
\item $g$ is {\em strongly Gauduchon,} if there is a $(n,n\!-\!2)$-form $\Psi$ such that $\partial\omega^{n-1}  =  \overline{\partial }\Psi$.
\item $g$ is {\em Gauduchon,} if $\partial \overline{\partial} \omega^{n-1} =0$.
\item $g$ is {\em Hermitian symplectic,} if there is a $(2,0)$-form $\alpha$ such that $\partial\alpha =0$ and $\partial \omega = - \overline{\partial }\alpha$.
\item $g$ is {\em pluriclosed,} if $\partial \overline{\partial} \omega =0$.
\end{enumerate}

The strongly Gauduchon condition was introduced by Popovici \cite{Popovici}. The condition $(iii)$ is not a restriction in the sense that, on a compact complex manifold, any Hermitian metric is conformal to a unique (up to constant multiple) Gauduchon metric.   Clearly, $(i)  \Rightarrow  (ii) \Rightarrow  (iii)$, and $(iv) \Rightarrow (v)$.

We begin with the following observation on a relationship between $(ii)$ and $(iv)$, which might be known to experts but should be of independent interest as well:

\begin{lemma}
Let $(M^n,g)$ be a Hermitian manifold that is Hermitian symplectic. Then $M^n$ admits a strongly Gauduchon metric $h$.
\end{lemma}

\begin{proof}
Write $\omega =\omega_g$ for the K\"ahler form of $g$. By definition, there is a $(2,0)$-form $\alpha$ on $M^n$ such that $\partial \alpha =0$ and $\partial \omega = - \overline{\partial }\alpha$. Consider the $d$-closed real $2$-form
$$ \chi = \alpha + \omega + \overline{\alpha}.$$
It follows easily that the $d$-closed, real, $(2n-2)$-form $\chi^{n-1}$ decomposes as
$$ \chi^{n-1} = \Psi + \Omega + \overline{\Psi}, $$
where $\Omega$ is the $(n\!-\!1,n\!-\!1)$-part and $\Psi$ the $(n,n\!-\!2)$-part. The analysis of the $(n,n\!-\!1)$-part of the form $d\chi^{n-1}=0$ yields
$$ \partial \Omega + \overline{\partial} \Psi = 0. $$
It is easy to see that $\alpha \overline{\alpha} \geq 0$, and for any $k\geq 1$, $\alpha^k \overline{\alpha}^k = (\alpha \overline{\alpha})^k \geq 0$. Note that only when $j$ is even, $(\alpha + \overline{\alpha})^j$ may contain $(p,p)$-components. Then we have
\begin{eqnarray*}
\Omega & = & \sum_{k\geq 0} C^{2k}_{n-1} \omega^{n-1-2k} C^k_{2k} (\alpha \overline{\alpha })^k   \\
& = & \omega^{n-1} + C^1_2C^2_{n-1} \omega^{n-3} \alpha \overline{\alpha } + C^2_4 C^4_{n-1} \omega^{n-5} (\alpha \overline{\alpha })^2 + \cdots  \\
& \geq & \omega^{n-1} \ > \ 0.
\end{eqnarray*}
As is well-known, any positive $(n\!-\!1,n\!-\!1)$-form  can be written as the $(n\!-\!1)$-th power of a positive $(1,1)$-form, therefore we have
a Hermitian metric $h$ on $M$ such that $\omega_h^{n-1} = \Omega$, and $h$ is strongly Gauduchon as
$$ \partial (\omega_h^{n-1}) = \partial \,\Omega = - \overline{\partial }\Psi . $$
\end{proof}

\begin{proof}[\textbf{Proof of Theorem \ref{thm1}}]
Suppose that $(M^n,g)$ is a compact non-K\"ahler SKL manifold. Denote by $\omega$ the K\"ahler form of $g$. It follows that $\partial \overline{\partial }\omega =0$ from \cite[Theorem 1]{ZZ}, which implies that it represents an Aeppli cohomology class $[\omega]_A$ in $H^{1,1}_A(M)$. If it is trivial, then there will be $(1,0)$-forms $\beta$, $\sigma$ on $M^n$ such that
$$\omega  = \overline{\partial } \sigma + \partial \overline{ \beta}.$$
Since $\overline{\omega}=\omega$, we may assume that $\beta = \sigma$. Consider the $(2,0)$-form $\alpha =\partial \sigma$. We have  $\partial \alpha =0$ and $\partial \omega = - \overline{\partial }\alpha$. Hence $\omega$ is Hermitian symplectic. By the above lemma, we know that $M^n$ admits a strongly Gauduchon metric $h$, contradicting with \cite[Proposition 3]{ZZ}. Therefore it follows that $[\omega]_A \neq 0$ in $H^{1,1}_A(M)$. Meanwhile, since $M^n$ can not admit any Hermitian symplectic metric, we know that $M^n$ does not satisfy the $\partial \overline{\partial}$-Lemma, as the $\partial \overline{\partial}$-Lemma would turn any pluriclosed metric into a Hermitian symplectic one. \end{proof}

Recall that Fu, Wang and Wu \cite{Fu-Wang-Wu1} introduced the notion of {\em $k$-Gauduchon} for Hermitian manifold $(M^n,g)$, where $k$ is a positive integer less than $n$. It is defined by $\partial \overline{\partial } (\omega ^k) \wedge \omega ^{n-k-1} =0$, where $\omega$ is the K\"ahler form of $g$. When $k=n-1$, this is just the original Gauduchon condition. They studied the existence problem for $k$-Gauduchon metrics within a conformal class of a compact manifold, which generalizes Gauduchon's classic result for the $k=n-1$ case. For each $1\leq k\leq n-2$, their results in particular imply that the $k$-Gauduchon metrics, if exists, are unique (up to constant multiples) within a conformal class. In \cite{ZZ}, we showed that any SKL metric is pluriclosed and Gauduchon, by what follows, it is also $k$-Gauduchon for any $k$.

\begin{remark}
Let $(M^n,g)$ be a Hermitian manifold such that $g$ is both pluriclosed and Gauduchon. Then $g$ is $k$-Gauduchon for any $1\leq k\leq n-1$.
In particular, a SKL metric $g$ is necessarily $k$-Gauduchon for any $1\leq k\leq n-1$.
\end{remark}

\begin{proof}
The proof is straight forward. Since $g$ is both pluriclosed and Gauduchon, it follows that
\[\partial \overline{\partial } \omega=0\quad \text{and}\quad \partial \omega \wedge \overline{\partial } \omega \wedge \omega^{n-3}=0.\]
For any integer $1\leq  k \leq n-1$, it yields that
\begin{eqnarray*}
\partial \overline{\partial } (\omega ^k) \wedge \omega ^{n-k-1} & = & k \partial (\overline{\partial }\omega \wedge \omega^{k-1}) \wedge \omega ^{n-k-1} \\
& = & k\partial \overline{\partial } \omega  \wedge \omega^{n-2} + k(k-1) \partial \omega \wedge \overline{\partial } \omega \wedge \omega^{n-3} \\
& = & 0.
\end{eqnarray*}
It follows from \cite[Proposition 3 and Remark 1]{ZZ} that a SKL metric is both pluriclosed and Gauduchon.
\end{proof}
Therefore, when $M^n$ is compact, we know that any SKL metric is unique (up to constant multiples) in its conformal class. The more interesting part is that, when $M^n$ is non-compact but $n\geq 3$, any SKL metric is still unique within its conformal class, which will be postponed to Theorem \ref{thm3}.

To prove the next a few theorems, let us assume that $(M^n,g)$ is a SKL manifold. Denote by $\nabla^s$ the Strominger connection, and by $T^j_{ik}$ the components of the Chern torsion under a unitary frame $e$. Let $\varphi$ be the dual coframe and $\eta$ be the Gauduchon torsion $1$-form \cite{Gauduchon}, defined by the identity $\partial \omega^{n-1} = -2\eta \,\omega^{n-1}$. We will use the same notations as in \cite{YZ} and \cite{ZZ}, and also \cite{Zheng} for as a general reference. It follows from \cite{ZZ} that $\nabla^sT=0$, $\partial \overline{\partial}\omega =0$, and
\begin{eqnarray}
&& \sum_r \eta_r T^r_{ik}=0 \label{etaT}\\
&& P_{ik}^{j\ell} := \sum_r \big\{ T_{ik}^r \overline{ T^r_{j\ell}} +   T^j_{ir} \overline{ T^k_{\ell r}} + T^{\ell}_{kr} \overline{ T^i_{j r}} - T^j_{kr} \overline{ T^i_{\ell r}} - T^{\ell}_{ir} \overline{ T^k_{j r}} \big\} =0 \label{eq:P}
\end{eqnarray}
for any indices $i$, $k$, $j$, $\ell$. Meanwhile, $B=\phi + \phi^{\ast}$ as shown in \cite[Lemma 10 and 12]{ZZ}, where
\begin{equation}\label{BPhi}
B_{ij} = \sum_{r,s} T^j_{rs} \overline{ T^i_{rs}}, \ \ \ \phi_i^j = \sum_r \overline{\eta}_r T^j_{ir},
\end{equation}
and $\eta$, $\phi$, $B$ are all parallel under $\nabla^s$.

\begin{proof}[\textbf{Proof of Theorem \ref{thm2}}]
Let $(M^n,g)$ be a non-K\"ahler SKL manifold. Consider the vector field
\[ X_{\!\eta} = \sum_r \overline{\eta}_r e_r\]
on $M^n$. It is easy to see that it is independent of the choice of the local unitary frame $e$, hence is globally defined. Let us show that $X_{\!\eta}$ is parallel with respect to the Strominger connection $\nabla^s$, namely, $\nabla^sX_{\!\eta}=0$. To see this, fix a point $x\in M$ and choose a local unitary frame $e$ in a neighborhood of $x$ such that the connection matrix $\theta^s$ of $\nabla^s$ vanishes at $x$. At the point $x$, it yields that
$$ \nabla^s_v X_{\!\eta} = v(\overline{\eta}_r) e_r = \overline{ \eta_{r, \overline{v}} } \,e_r = 0, $$
since $\nabla^s\eta =0$, where $v$ is any $e_i$ or $\overline{e}_i$.

Next we show that $X_{\!\eta}$ must be a holomorphic vector field, which means that $\nabla^c_{\overline{e}_i} X_{\!\eta} =0$ for any $i$, where $\nabla^c$ is the Chern connection. At the point $x$, since $\theta^s=0$, the connection matrix $\theta$ for $\nabla^c$ is equal to $-2\gamma$, where
$$ \gamma_{ij} = \sum_k \big\{ T^j_{ik}\varphi_k - \overline{T^i_{jk} } \,\overline{\varphi}_k \big\} .$$
At the point $x$,  the structure equation gives us
$$ d \varphi = -\,^t\!\theta \varphi + \tau = 2\,^t\!\gamma \varphi +\tau = - \tau - 2 \overline{\gamma '} \varphi.$$
Here we have used the fact that $^t\!\gamma'\varphi = - \tau$, where $\gamma'$ is the $(1,0)$-part of $\gamma$. It follows that, at $x$,
\begin{eqnarray}
&& \partial \varphi_r = - \tau_r = - \sum_{i,k} T^r_{ik} \varphi_i \varphi_k  \label{dvarphi}\\
&& \overline{\partial} \varphi_r = - 2 \overline{\gamma'_{rk}} \,\varphi_k =  -2 \sum_{k,j} \overline{ T^k_{rj}}\, \overline{\varphi}_j \varphi_k \label{dbarvarphi}
\end{eqnarray}
Since $\nabla^sX_{\!\eta}=0$, it follows that, at $x$,
$$ \nabla^c_{\overline{e}_i} X_{\!\eta} = -2 \overline{\eta}_r \gamma_{rk}(\overline{e}_i) e_k = 2 \overline{\eta}_r \,\overline{ T^r_{ki} } \,e_k  =0  $$
for any $i$, where the last equality is due to (\ref{etaT}). This shows that $X_{\eta}$ is a holomorphic vector field, therefore we have completed the proof of Theorem \ref{thm2}.
\end{proof}

The covariant derivative of $X_{\!\eta}$ with respect to the Riemannian connection $\nabla$ will be calculated for the later use.

\begin{lemma}
Let $(M^n,g)$ be a non-K\"ahler SKL manifold and $X_{\!\eta}$ be the vector field dual to the torsion $1$-form $\eta$. Then under any unitary frame $e$, it yields that
\begin{eqnarray}
\nabla_{e_i} X_{\!\eta} & = & - \sum_{r,k}    \overline{\eta}_r T^k_{ri} e_k  = \sum_k \phi_i^k e_k \\
\nabla_{\overline{e}_i} X_{\!\eta} & = & \sum_{r,k}  \overline{\eta}_r ( \overline{T^r_{ki}} \, e_k + T^i_{rk} \overline{e}_k ) = - \sum_k \phi_k^i \overline{e}_k
\end{eqnarray}
\end{lemma}

\begin{proof} This is the direct consequence of $\nabla^s X_{\!\eta}=0$ and the fact that
$$ \nabla e_r = \nabla^s e_r - \sum_k \gamma_{rk} e_k + \sum_{k,i} T^i_{rk} \overline{\varphi}_i \,\overline{e}_k,$$
together with the equality (\ref{etaT}).
\end{proof}

Let us discuss the uniqueness problem for SKL metrics within a conformal class.

\begin{proof}[\textbf{Proof of Theorem \ref{thm3}}]
Let $(M^n,g)$ be a SKL manifold, where $n\geq 3$. Suppose that $\tilde{g}=e^{2u}g$ is another SKL metric conformal to $g$, where $u$ is real valued smooth function on $M^n$. We want to show that $u$ must be constant. Denote by $\omega$, $\tilde{\omega } = e^{2u}\omega$ the K\"ahler form of $g$, $\tilde{g}$, respectively. Let $\varphi$ be a local unitary coframe for $g$ and $\theta^s$ be the connection matrix of the Strominger connection $\nabla^s$ of $g$ under $\varphi$, with $\tilde{\varphi} = e^u \varphi$, which is the associated unitary coframe of $\tilde{g}$, and $\tilde{\theta}^s$ being the $\nabla^s$-matrix under $\tilde{\varphi}$. Let us also denote by $e$ the local unitary frame of $g$ dual to $\varphi$, with $\tilde{e} = e^{-u}e$, which is the associated local unitary frame of $\tilde{g}$ dual to $\tilde{\varphi}$. Since
\begin{eqnarray*}
 d\tilde{\varphi} & = &  du \,\tilde{\varphi} + e^u(-\,^t\!\theta \varphi + \tau ) \\
 & = & \{ (\overline{\partial }u - \partial u) I - \,^t\!\theta \} \tilde{\varphi} + \{ 2\partial u \tilde{\varphi} + e^u \tau \} \\
 & = &  - \,^t\!\tilde{\theta } \tilde{\varphi} + \tilde{\tau} ,
\end{eqnarray*}
where $\theta$ and $\tau$ are the connection matrix and column vector of the torsion of the Chern connection of $g$ under $\varphi$ respectively, with $\tilde{\theta}$ and $\tilde{\tau}$ being those of $\tilde{g}$ under $\tilde{\varphi}$.
It follows that
$$ \tilde{\theta} = \theta + (\partial u - \overline{\partial }u)I  \quad \text{and} \quad \tilde{\tau } = e^u(\tau + 2 \partial u\,\varphi).$$
From this, it yields that
\begin{eqnarray*}
\tilde{T}^j_{ik} & = & e^{-u} \{ T^j_{ik} + u_i \delta_{jk} - u_k \delta_{ji} \}, \\
\tilde{\eta}_k & = & e^{-u} \{ \eta_k - (n-1)u_k \},
\end{eqnarray*}
where $u_k=e_k(u)$. With $P=\tilde{\theta}^s-\theta^s$ denoted by the difference of the Strominger connection matrix of $\tilde{g}$ and $g$ under the respective unitary coframes $\tilde{\varphi}$ and $\varphi$, it follows that
\begin{eqnarray*}
P_{ik} & = & \tilde{\theta}_{ik} - \theta_{ik} + 2 \tilde{\gamma}_{ik} - 2 \gamma_{ik} \\
& = & ( \partial u - \overline{\partial } u ) \delta_{ik} + 2( \tilde{T}^k_{ij} \tilde{\varphi}_j - \overline{ \tilde{T}^i_{kj} \tilde{\varphi}_j     } ) - 2 ( T^k_{ij} \varphi_j - \overline{ T^i_{kj} \varphi_j } ) \\
& = & 2u_i \varphi_k - \partial u \delta_{ik} - 2 \overline{u}_k \overline{\varphi}_i + \overline{\partial} u \delta_{ik}.
\end{eqnarray*}
As in \eqref{BPhi} above, it holds that $B_{i\overline{j}} = \phi_i^j + \overline{\phi_j^i}$ for a SKL metric.
Then the following equality is established
\begin{eqnarray*}
\tilde{\phi}_i^j  & = & \tilde{T}^j_{ir} \overline{ \tilde{\eta}_r } \ = \ e^{-2u} ( T^j_{ir} + u_i\delta_{jr} - u_r\delta_{ij}) \,( \overline{\eta}_r - (n-1) \overline{u}_r )\\
& = & e^{-2u} \{ \phi_i^j + u_i \overline{\eta}_j - u_r\overline{\eta}_r \delta_{ij}
- (n-1) T^j_{ir} \overline{u}_r - (n-1)u_i \overline{u}_j + (n-1) |u_r|^2 \delta_{ij} \},
\end{eqnarray*}
which yields that
\begin{eqnarray*}
e^{2u} \tilde{B}_{i\overline{j}}  & = & e^{2u} ( \tilde{\phi}^j_i + \overline{ \tilde{\phi}^i_j  } )  \ = \
B_{i\overline{j}} + ( u_i \overline{\eta}_j + \eta_i \overline{u}_j) -2(n-1) u_i \overline{u}_j \\
& & - (n-1)(  T^j_{ir} \overline{u}_r + \overline{ T^i_{jr} } u_r )  - (u_r\overline{\eta}_r + \eta_r \overline{u}_r ) \delta_{ij} + 2 (n-1) |u_r|^2 \delta_{ij}.
\end{eqnarray*}
On the other hand, the definition of $\tilde{B}_{i\overline{j}} $ leads to
\begin{eqnarray*}
e^{2u} \tilde{B}_{i\overline{j}}  & = & e^{2u} \,\tilde{T}^j_{rs} \overline{ \tilde{T}^i_{rs}}  \ = \ (T^j_{rs} + u_r\delta_{js} - u_s \delta_{jr} ) \, ( \overline{T^i_{rs}} + \overline{u}_r \delta_{is} - \overline{u}_s \delta_{ir} ) \\
& = & B_{i\overline{j}} - 2 u_r \overline{ T^i_{jr} } - 2 \overline{u}_r T^j_{ir} - 2 u_i \overline{u}_j + 2 |u_r|^2 \delta_{ij}.
\end{eqnarray*}
By the comparison of the above two expressions, it yields that
\begin{equation}
 (u_i \overline{\eta}_j + \eta_i \overline{u}_j) -2(n-2)u_i \overline{u}_j = (n-3)\sum_r ( T^j_{ir} \overline{u}_r + \overline{ T^i_{jr} } u_r ) + \delta_{ij} \sum_r ( u_r \overline{\eta}_r + \eta_r \overline{u}_r - 2(n-2) |u_r|^2 ). \label{eq:BBcompare}
\end{equation}
Note that when both $g$ and $\tilde{g}$ are K\"ahler, the conformal factor $u$ is necessarily constant. Hence, we may assume that $g$ is non-K\"ahler. The $\nabla^s$-parallelness of $\eta$ enables us to choose the local unitary frame $e$ such that
\[\eta_i=0\quad \text{for} \quad 1\leq i\leq n-1\quad \text{and}\quad \eta_n =\lambda>0\quad \text{where}\ \lambda\ \text{is a global constant}.\]
It follows from \eqref{etaT} that $T^n_{ik}=0$ for any indices $i$, $k$. After $i=j=n$ is set in the identity \eqref{eq:BBcompare}, it yields that
$$ \lambda (u_n + \overline{u}_n) - 2(n-2) |u_n|^2 = \lambda (u_n + \overline{u}_n) - 2 (n-2) \sum_r |u_r|^2 , $$
or equivalently,
$$ 2(n-2) \sum_{r=1}^{n-1} |u_r|^2 = 0. $$
It can be concluded, for $n\geq 3$,
\begin{equation}
u_1 = u_2 = \cdots = u_{n-1}=0. \label{eq:onlyun}
\end{equation}
When (\ref{eq:onlyun}) is plugged into (\ref{eq:BBcompare}) and $1\leq i=j\leq n-1$ is set, it follows that
\begin{equation*}
0 = (n-3) ( T^i_{in} \overline{u}_n + \overline{T^i_{in} } u_n ) +   \lambda (u_n + \overline{u}_n) -2(n-2) |u_n|^2.
\end{equation*}
After $i$ is summed up from $1$ to $n-1$, it leads to
\begin{equation}
\lambda (u_n + \overline{u}_n) = (n-1) |u_n|^2. \label{eq:un}
\end{equation}

The next step is to prove that $u_n=0$, hence $u$ must be a constant. Actually, the equality $\nabla^s\eta=0$ implies
$$ \eta_{i,j} = e_j(\eta_i) - \eta_r \,\theta^s_{ir}(e_j) =0, \ \ \ \eta_{i,\overline{j}} = \overline{e}_j(\eta_i) - \eta_r \,\theta^s_{ir}(\overline{e}_j) =0,$$
with similar equalities established for $\tilde{\eta}$. For $p \in M$, after the unitary frame $e$ such that $\theta^s=0$ at $p$ is applied and the previous expression for $P =  \tilde{\theta}^s - \theta^s$ is used, it yields that
\begin{eqnarray*}
0 & = & \tilde{\eta}_{i,j} \ = \ \tilde{e}_j (\tilde{\eta}_i) - \tilde{\eta}_r \,\tilde{\theta}^s_{ir}(\tilde{e}_j) \\
& = & e^{-u} \{ e_j \big( \tilde{\eta}_i \big) - \tilde{\eta}_r  \,\tilde{\theta}^s_{ir}(e_j) \} \\
& = & e^{-u}  e_j \{ e^{-u} ( \eta_i - (n-1)u_i)\}  - e^{-2u} \{ \eta_r -(n-1)u_r\} \, \{ 2u_i \delta_{jr} - u_j \delta_{ir}\} \\
& = & e^{-2u} \{ - u_j (\eta_i -(n-1)u_i) + \eta_{i,j} - (n-1)u_{i,j} - 2u_i(\eta_j-(n-1)u_j) + u_j(\eta_i-(n-1)u_i) \} \\
& = & e^{-2u} \{ - (n-1)u_{i,j} -2u_i ( \eta_j-(n-1)u_j) \},
\end{eqnarray*}
which implies that
\begin{equation}
u_{i,j} = 2 u_iu_j - \frac{2}{n-1}u_i\eta_j = 2u_n\delta_{in}\delta_{jn}(u_n-\frac{\lambda}{n-1}), \label{eq:uij}
\end{equation}
where the index $j$ after the comma in $u_{i,j}$ denotes the covariant derivative with respect to $\nabla^s$. Similarly, the calculation of $\tilde{ \eta}_{i, \overline{j}} $ yields
\begin{equation}
u_{i,\overline{j}} = 2 \overline{u}_n ( u_n -\frac{\lambda}{n-1} ) \, (\delta_{in} \delta_{jn} - \delta_{ij} ) . \label{eq:uijbar}
\end{equation}
With $i=n$ set in (\ref{eq:uijbar}), it follows that $u_{n\overline{j}}=0$ for any $j$. It also holds that $u_{n,n}=2u_n(u_n-\frac{\lambda}{n-1})$, after $i=j=n$ is set in (\ref{eq:uij}). Let $c=\frac{\lambda }{n-1}>0$ and one has, by (\ref{eq:un}), $c\,(u_n+\overline{u}_n)=|u_n|^2$,
which implies that \[cu_{n,n}=u_{n,n}\overline{u}_n.\]
Note that $u_n \neq c$ at each point, since it would lead to a contradiction that $2c^2=c^2$, from (\ref{eq:un}). It follows that
$u_{n,n}$ identically vanishes, which yields the same holds for $u_n$. Hence $u$ is a constant. Therefore we have completed the proof of Theorem \ref{thm3}.
\end{proof}

As the proof above is  local in nature, it yields that

\begin{remark}
Let $(M^n,g)$ be a SKL manifold with $g$ not K\"ahler. If $n\geq 3$, then $g$ is never locally conformal K\"ahler.
\end{remark}

\begin{proof}[\textbf{Proof of Theorem \ref{thm4}}]
Let $(M^n,g)$ be a non-K\"ahler SKL manifold with $n\geq 3$, where $g$ itself cannot be locally conformal K\"ahler by the remark above. 
To prove the theorem, let us assume the contrary that $M^n$ does admit a Vaisman metric $\hat{g}$. Denoted by $\omega$ and $\hat{\omega}$ the K\"ahler form of $g$ and $\hat{g}$, respectively.  By the very definition, $\hat{\omega}$ is a locally conformal K\"ahler metric, whose (real) Lee form $\hat{\vartheta}$ is $\nabla$-parallel, where $\nabla$ is the Riemannian connection of $\hat{g}$. Since $\hat{g}$ is locally conformal K\"ahler, we have $\hat{\vartheta}^{1,0}= \frac{1}{n-1}\hat{\eta}$ and
\begin{equation}\label{Teta}
\hat{T}^j_{ik} = \frac{1}{n-1} \big( \delta_{ji} \hat{\eta}_k - \delta_{jk} \hat{\eta}_i \big),
\end{equation}
where $\hat{T}^j_{ik}$ and $ \hat{\eta}_k$ are the components of the Chern torsion $\hat{T}$ and the Gauduchon's torsion $1$-form $\hat{\eta}$ of $\hat{\omega}$ under a unitary frame $e$ with respect to $\hat{\omega}$, with the dual coframe denoted by $\varphi$. From \cite[Lemma 7]{ZZ}, the $\nabla$-parallelenss of $\hat{\vartheta}$ is equivalent to the equalities
\[\begin{cases} \hat{\eta}_{i,k}   =   -\hat{\eta}_r \hat{T}^r_{ik}, \\
\hat{\eta}_{i,\overline{k}}   =  \overline{\hat{\eta}}_r \hat{T}^{k}_{ir} - \hat{\eta}_r \overline{\hat{T}^i_{k r}},
\end{cases}\]
for any $i$, $k$, where the index after comma means covariant derivative with respect to the Strominger connection $\nabla^s$ of $\hat{g}$.
It yields from \eqref{Teta} that $\hat{\eta}_{i,k}=0$, $\hat{\eta}_{i,\overline{k}}=0$, that is, $\nabla^s \hat{\eta}=0$, $\nabla^s \hat{T}=0$.
By the definition of $\hat{\eta}$, it follows that for $n \geq 3$,
\[ \begin{aligned}
\overline{\partial} \partial \hat{\omega}^{n-2} &= (n-2) \overline{\partial} (\partial \hat{\omega} \wedge \hat{\omega}^{n-3}) \\
&= - 2(n-2)  \overline{\partial} ( \hat{\vartheta}^{1,0} \wedge \hat{\omega}^{n-2}) \\
&= - \frac{2(n-2)}{n-1}  \overline{\partial} ( \hat{\eta} \wedge \hat{\omega}^{n-2}) \\
&= - \frac{2(n-2)}{n-1} (\overline{\partial} \hat{\eta} + \frac{2(n-2)}{n-1} \hat{\eta} \wedge \overline{\hat{\eta}}) \wedge \hat{\omega}^{n-2}. \\ \end{aligned}\]
Since $g$ is SKL and thus $\partial \overline{\partial} \omega=0$ by \cite{ZZ}, it forces that
\begin{equation}\label{pluric}
\begin{aligned}
0 & = \int_{M} \partial \overline{\partial} \omega \wedge \hat{\omega}^{n-2} \\
& = \int_{M} \partial (\overline{\partial} \omega \wedge \hat{\omega}^{n-2}) + \overline{\partial}(\omega \wedge \partial \hat{\omega}^{n-2})
- \omega \wedge \overline{\partial} \partial \hat{\omega}^{n-2} \\
& = \int_{M}  \frac{2(n-2)}{n-1} (\overline{\partial} \hat{\eta} + \frac{2(n-2)}{n-1} \hat{\eta} \wedge \overline{\hat{\eta}}) \wedge \hat{\omega}^{n-2} \wedge \omega.
\end{aligned}\end{equation}
From $\nabla^s \hat{\eta}=0$ and the equality \eqref{Teta}, it yields that $|\hat{\eta}|^2 =\sum_{r}|\hat{\eta}_r|^2$ is a constant, denoted by $\lambda^2$ with $\lambda \geq 0$, and
\[\begin{aligned}
\overline{\partial} \hat{\eta} &= \sum_{i,j}-(\hat{\eta}_{i,\bar{j}} + 2 \sum_r \hat{\eta}_r\overline{\hat{T}^i_{jr}}) \varphi_i \wedge \overline{\varphi}_j \\
&= \frac{2}{n-1}( \sqrt{-1} \lambda^2 \hat{\omega} + \hat{\eta} \wedge \overline{\hat{\eta}}).
\end{aligned}\]
After a possible unitary change of the frame $\varphi$, still denoted by $\varphi$, it can be assumed that
\[ \omega = \sqrt{-1} \sum_i \lambda_i \varphi_i \wedge \overline{\varphi}_i,\]
where $\{\lambda_i\}_{i=1}^n$, globally defined real positive continuous functions on $M^n$, are the eigenvalues of $\omega$ with respect to $\hat{\omega}$. Then it follows clearly that
\[ \sqrt{-1} (\overline{\partial} \hat{\eta} + \frac{2(n-2)}{n-1} \hat{\eta} \wedge \overline{\hat{\eta}}) \wedge \hat{\omega}^{n-2} \wedge \omega
=-2 \left(\sum_r \lambda_r |\hat{\eta}_r|^2 \right)\frac{\hat{\omega}^n}{n(n-1)},\]
which yields by \eqref{pluric} that $\hat{\eta}_r=0$ for any $r$, and thus $\hat{\omega}$ is a balanced metric. However, a balanced metric can never appear on a compact non-K\"ahler SKL manifold by \cite[Theorem 3]{ZZ}. This completes the proof of Theorem \ref{thm4}.
\end{proof}

\section{SKL manifolds with degenerate torsion}\label{degtor}

Let us focus on the SKL manifolds that are three dimensional, or more generally, when its torsion tensor has a lot of degeneracy and thus some terminologies will be introduced.

\begin{definition}
Let $(M^n,g)$ be a non-K\"ahler SKL manifold. A local unitary frame $e$ is said to be \textbf{admissible}, if $X_{\!\eta}=\lambda e_n$ for $\lambda >0$ and under $e$ the matrix $ \phi = (\phi_i^j)$ is diagonal.
\end{definition}

Note that $\lambda = |\eta| = |X_{\!\eta}|$ is a positive constant, and $e_n$ is globally defined, but for $1\leq i\leq n-1$, $e_i$ is only locally defined, so such frames are not uniquely determined in general. First we claim that such frames always exist locally:

\begin{lemma}
Let $(M^n,g)$ be a non-K\"ahler SKL manifold. For any $x\in M$, there always exists an admissible frame $e$ in a neighborhood of $x$.
\end{lemma}

\begin{proof}
For any given $x\in M$, let $e$ be a local unitary frame with $e_n=\frac{1}{\lambda} X_{\!\eta}$, where $\lambda = |\eta|$. Under the frame $e$, it follows that $\eta_i=0$ for each $i<n$, $\eta_n=\lambda$, and $T^n_{\ast \ast}=0$ by (\ref{etaT}). Let us take $k=\ell =n$ in (\ref{eq:P}) and multiply the equality by $\lambda^2$, which yields that
$$ \sum_r \{ \phi_i^r \overline{ \phi_j^r }- \phi_r^j \overline{ \phi_r^i  }  \} =0 $$
for any $i$, $j$. That is, the matrix $\phi = (\phi_i^j)$ satisfies $\phi \phi^{\ast} = \phi^{\ast} \phi$ and thus it is normal. Hence by a unitary change of $\{ e_1 , \ldots , e_{n-1}\}$ with $e_n$ fixed, we can make $\phi$ diagonal, since $\phi_n^j=\phi_i^n=0$, and thus obtain a local unitary frame that is admissible.
\end{proof}

\begin{remark}\label{admissibleT}
Let $(M^n,g)$ be a non-K\"ahler SKL manifold. Under an admissible frame $e$, it yields that $T^n_{\ast\ast}=0$ and $T^j_{in}=0$ for $i\neq j$.
\end{remark}

\begin{definition}\label{degtor_def}
A non-K\"ahler SKL manifold $(M^n,g)$ is said to have \textbf{degenerate torsion}, if under any admissible frame $e$, $T^{\ast }_{ik}=0$ for any $i,k <n$.
\end{definition}

\begin{remark}\label{degenerateT}
Under an admissible frame $e$ of a non-K\"ahler SKL manifold $(M^n,g)$ with degenerate torsion, the only possibly non-zero components of the torsion $T$ are $T^i_{in}$ for $i \leq n-1$.
\end{remark}

From now on, let us write $a_i=T^i_{in}$. Then under any admissible frame $e$, we always have
$$\sum_i a_i=\lambda \quad \mbox{and} \quad \phi_i^j = \lambda a_i \delta_{ij}\ \ \mbox{with}\ \ a_n=0.$$
If the torsion $T$ degenerates, then by letting $i=j < k=\ell < n$ in (\ref{eq:P}), we get
$$ a_i \overline{a}_k + a_k \overline{a}_i = 0, \ \ \ \forall \  1\leq i < k \leq n-1. $$
That is, the vector $\begin{pmatrix} a_i \\ \overline{a_i}\end{pmatrix}$ is orthogonal to $\begin{pmatrix} a_k \\ \overline{a_k}\end{pmatrix}$ for any $1\leq i\neq k \leq n-1$. So the set $\{ a_1, \ldots , a_{n-1}\}$ can have at most two non-zero elements.

\begin{lemma} \label{lemma6}
Let $(M^n,g)$ be a non-K\"ahler SKL manifold. It follows that under an admissible frame $e$, $a_i$ is a globally defined constant for each $1\leq i \leq n$. When $(M^n,g)$ has degenerate torsion, there always exists an admissible frame $e$ such that one of the following two cases occur:
\begin{enumerate}
\item the rank 1 case:\quad $a_1 = \cdots = a_{n-2}=0, \quad a_{n-1}=\lambda$,
\item the rank 2 case:\quad $a_1 = \cdots = a_{n-3}=0, \quad a_{n-2} = \frac{\lambda}{2}(1+ \rho ), \quad a_{n-1} = \frac{\lambda}{2}(1-\rho )$,
\end{enumerate}
where $\rho $ is a globally defined constant with positive imaginary part and with $|\rho |=1$.
\end{lemma}

\begin{proof}
It is clear that $\{a_i\}_{i=1}^n$ are the eigenvalues of $\phi$ under an admissible frame $e$ and unitary transformation
between admissible frames don't change eigenvalues, which implies that $\{a_i\}_{i=1}^n$ are globally defined functions.
The $\nabla^s$-parallelness of the tensor $\phi$ forces $\{a_i\}_{i=1}^n$ to be constants.

When $(M^n,g)$ has degenerate torsion, we already know that there can be at most two non-zero elements amongst those $a_i$. Since their sum is $\lambda$, it follows that either one of them is $\lambda$ while the rest are zero, or exactly two of them are non-zero. In the latter case, since their sum is $\lambda$ and the real part of one times the conjugate of the other vanishes, they must be in the above given form for some number $\rho$ with norm $1$. Through a permutation of the first $n-1$ elements of $e$ if necessary, we get the required presentation. The conclusion that $\rho$ is a global constant, independent of the choice of the local frame, follows from that of $\{a_i\}_{i=1}^n$.
\end{proof}

When $(M^n,g)$ is a non-K\"ahler SKL manifold with degenerate torsion, for an admissible frame $e$ satisfying the requirement of the above lemma, the two cases are separated by the rank of the matrix $\phi $, which is either $1$ or $2$. Hence we will call these two cases respectively the {\em rank $1$} or {\em rank $2$} case.

Denote by $E$ the $\nabla^s$-parallel distribution in $T^{1,0}M$ generated by $\{ e_1, \ldots , e_{n-2}\}$ in the rank $1$ case or $\{ e_1, \ldots , e_{n-3}\}$ in the rank $2$ case, respectively. We will call $E$ the {\em kernel distribution}  of the torsion $T$. Clearly, $E$ is $\nabla^s$-parallel. Let $m=n-2$ in the rank $1$ case and $m=n-3$ in the rank $2$ case. Later in the proof of Theorem \ref{thm10}, we will show that $E \oplus \overline{E}$ is parallel with respect to the Riemannian connection $\nabla$, hence at the universal covering level it gives the de Rham decomposition and splits off a K\"ahler factor of complex dimension $m$.

\begin{definition}\label{sadmb}
For a non-K\"ahler SKL manifold $(M^n,g)$ with degenerate torsion, we will call an admissible local frame $e$ \textbf{strictly admissible}, if under $e$ the components of $\phi$ takes the special form as in Lemma \ref{lemma6} above.
\end{definition}

\begin{remark}\label{sadmissibleT}
It follows that under an strictly admissible frame, the non-zero components of the torsion $T$ are $T^i_{in}$ for $m < i \leq n-1$,
where $m$ is defined right before the definition above.
\end{remark}

We observe that for a non-K\"ahler SKL manifold $(M^n,g)$, under an admissible frame $e$, the connection matrix $\theta^s$ for $\nabla^s$ is block diagonal according to different eigenvalues of $\phi$. This is clear since $\phi$ is $\nabla^s$-parallel, so the eigenspaces for distinct eigenvalues of $\phi$ form $\nabla^s$-parallel distributions in $T^{1,0}M$. Alternatively, we may consider the covariant derivatives of the torsion under $\nabla^s$, which yields that
$$ 0= T^j_{ik, \ell} = e_{\ell}(T^j_{ik}) + \sum_r \{ T^j_{rk } \,\theta^s_{ir}(e_{\ell}) + T^j_{ir } \,\theta^s_{kr}(e_{\ell}) - T^r_{ik } \,\theta^s_{rj}(e_{\ell}) \} .$$
Let $k=n$, we get $(a_j-a_i)\, \theta^s_{ij}(e_{\ell}) =0$. Similarly, by using $T^j_{ik, \overline{\ell}} =0$, we get $(a_j-a_i) \, \theta^s_{ij}(\overline{e}_{\ell}) =0$. Therefore, it follows that $(a_j-a_i) \,\theta^s_{ij} =0$, which implies that $\theta^s_{ij}=0$ whenever $a_i\neq a_j$, $1\leq i, j\leq n-1$. Note that by our choice of $e$, $\theta^s_{nj}=\theta^s_{jn}=0$ for all $j$. This also shows that $\theta^s$ is block diagonal. In the special case when the torsion is degenerate, we can go one step further:

\begin{lemma}\label{lemma7}
If $(M^n,g)$ is a non-K\"ahler SKL manifold with degenerate torsion, then locally there exists a strictly admissible frame $e$ so that under $e$ the connection and curvature matrices of $\nabla^s$ are block diagonal in the following form,
\begin{enumerate}
\item the rank $1$ case:
\[\theta^s = \left[ \begin{array}{ccc} \ast_{n\!-\!2} &  & \\  & \alpha  & \\ & & 0 \end{array} \right] ,\ \ \  \ \ \ \ \Theta^s =  \left[ \begin{array}{ccc} \star_{n\!-\!2} &  & \\  & d\alpha  & \\ & & 0 \end{array} \right], \]
\item the rank $2$ case:
\[ \theta^s = \left[ \begin{array}{cccc} \ast_{n\!-\!3} &  & & \\  & \beta  & & \\    && \alpha  & \\ & & & 0 \end{array} \right] , \ \ \ \ \ \ \ \Theta^s = \left[ \begin{array}{cccc} \star_{n\!-\!3} &  & & \\  & d\beta  & & \\    && d\alpha  & \\ & & & 0 \end{array} \right],  \]
\end{enumerate}
where $\alpha$ and $\beta$ are local $1$-forms satisfying $\alpha +\overline{\alpha}=0$, $\beta + \overline{\beta}=0$, and
\begin{equation} d\alpha = f \varphi_{n\!-\!1} \wedge \overline{\varphi}_{n\!-\!1}, \ \ \ d\beta = h \varphi_{n\!-\!2} \wedge \overline{\varphi}_{n\!-\!2}, \label{eq:Theta}
\end{equation}
for some local real valued functions $f$ and $h$.
\end{lemma}

\begin{proof}
We already see that $\theta^s$ is block diagonal and thus $\Theta^s=d\theta^s-\theta^s\wedge \theta^s$ is also block diagonal. The only thing we need to verify here is the format of the curvature (\ref{eq:Theta}). To see this, note that $d\alpha$ is a $2$-form, satisfying $d\alpha + \overline{d\alpha}=0$, and so is $d\beta$. Since the Strominger connection $\nabla^s$ is K\"ahler-like, it means that $\,^t\!\varphi \,\Theta^s=0$, hence $\varphi_{n\!-\!1}\wedge d\alpha =0$. It yields that the $(0,2)$-part of $d\alpha$ must vanish, and since $d\alpha = - \overline{ d\alpha}$, its $(2,0)$-part also vanishes. For the $(1,1)$-part, write
$$ d\alpha = \sum_{i,j} A_{ij} \varphi_i \wedge \overline{\varphi}_j, $$
the coefficient matrix $A$ is Hermitian. The K\"ahler-like condition forces $A_{ij}=0$ for any $i\neq n\!-\!1$, and thus it has only one possibly non-zero element at the $(n\!-\!1,n\!-\!1)$-th position. Therefore $d\alpha$ takes the desired form. The same goes with $d\beta$. This completes the proof of the lemma.
\end{proof}

As to the $n=3$ case, it follows that

\begin{lemma}\label{lemma8}
Let $(M^n,g)$ be a non-K\"ahler SKL manifold of dimension $n\leq 3$. Then it has degenerate torsion.
\end{lemma}

\begin{proof}
 The $n=2$ case is automatic, so let us assume that $n=3$. Let $e$ be an admissible frame. It yield that $\eta_1=\eta_2=0$, $\eta_3 = \lambda$, and $T^3_{\ast \ast}=0$. What we need to show is that $T^j_{12}=0$ for $j=1,2$. It follows that
$$ 0 = \eta_2 = T^1_{12} + T^2_{22} + T^3_{32} = T^1_{12}. $$
Similarly, $T^2_{12}=0$. Therefore $M^3$ has degenerate torsion.
\end{proof}

If the two lemmata above are combined, we get the proof of Theorem \ref{thm6}:

\begin{proof}[\textbf{Proof of Theorem \ref{thm6}}]
Let $(M^n,g)$ be a non-K\"ahler SKL manifold with $n\leq 3$. The above lemma says that it has degenerate torsion, hence the values $\{ a_1, \ldots , a_{n-1}\}$ are given by: $a_1=\lambda$ when $n=2$, and either $a_1=0$, $a_2=\lambda$ or $a_1=\frac{\lambda}{2}(1+\rho)$, $a_2=\frac{\lambda}{2}(1-\rho)$ when $n=3$. In each case, the connection matrix $\theta^s$ is diagonal. Hence the curvature matrix $\Theta^s$ is also diagonal, with its $(i,i)$-entry given by
$ f_i \varphi_i \overline{\varphi}_i$ for some real function $f_i$, since $\,^t\!\varphi \,\Theta^s=0$. In particular, the Ricci curvature of $\nabla^s$ takes the form
$$ \mbox{tr}\Theta^s = \sum_{i=1}^n f_i \varphi_i \overline{\varphi}_i $$
where $n=2$ or $3$. Therefore when $\mbox{tr}\Theta^s=0$, we see that each $f_i=0$ and thus $\Theta^s=0$. This has completed the proof of Theorem \ref{thm6}.
\end{proof}

Note that for a K\"ahler manifold of dimension $2$ or higher, the vanishing of the Ricci curvature certainly does not imply the vanishing of the curvature. In order to generalize Theorem \ref{thm6} to higher dimensions, one needs to at least remove the K\"ahler de Rham factors (of dimension $\geq 2$) contained in the non-K\"ahler SKL manifold $M^n$.

Next, let us prove the main result of this paper, Theorem \ref{thm8} stated in the introduction section.

\begin{proof}[\textbf{Proof of Theorem \ref{thm8}}]
Let us start with a non-K\"ahler SKL manifold $(M^3,g)$. By our previous lemmata, there is a global holomorphic vector field $e_3$ on $M^3$, such that $X_{\!\eta} = \lambda e_3$ with $\lambda >0$. Also, we know that it has degenerate torsion, and locally there exists a unitary frame $e$ extending $e_3$, the so-called strictly admissible frame, such that $\phi$ is diagonal, with $\phi_1^1=\lambda a$, $\phi_2^2 =\lambda b$, where the two cases of the \emph{rank 1 and 2} as in Lemma \ref{lemma6} are divided:
\begin{enumerate}
\item\label{case1} $a=0, \quad b =\lambda$;
\item\label{case2} $a = \frac{\lambda}{2}(1+\rho),\quad b = \frac{\lambda}{2}(1-\rho)$,\quad for a globally defined constant $\rho $ with $|\rho|=1$ and $\mathrm{Im}(\rho)>0$.
\end{enumerate}
Note that $T^1_{13}=a$ and $T^2_{23}=b$ are the only non-zero components of the torsion tensor, and in both cases, we have $a\neq b$. The following lemma is actually a special case of Lemma \ref{lemma7}, for which the proof is omitted.

\begin{lemma}\label{lemma9}
Let $(M^3,g)$ be a non-K\"ahler SKL manifold. Then under a strictly admissible frame $e$, the connection matrix for the Strominger connection $\nabla^s$ is diagonal:
\begin{equation}
 \theta^s = \left[ \begin{array}{ccc} \sigma_1 & 0 & 0 \\ 0 & \sigma_2 & 0  \\  0 & 0 & 0  \end{array} \right]
 \end{equation}
 where $\sigma_1 + \overline{\sigma}_1=0$ and  $\sigma_2 + \overline{\sigma}_2=0$.
\end{lemma}

As a consequence, the connection matrix for the Riemannian connection $\nabla$ follows:

\begin{lemma}\label{lemma10}
Let $(M^3,g)$ be a non-K\"ahler SKL manifold. Then under a strictly admissible frame $e$, the Riemannian connection $\nabla$ takes the form:
\begin{eqnarray}
\nabla e_1 & = & \sigma_1' e_1 - \overline{a\varphi}_1 e_3 + a \,\overline{\varphi}_1 \overline{e}_3  \label{eq:e1}\\
\nabla e_2 & = & \sigma_2' e_2 - \overline{b\varphi}_2 e_3 + b \,\overline{\varphi}_2 \overline{e}_3 \label{eq:e2}\\
\nabla e_3 & = & a \varphi_1 e_1 + b \varphi_2 e_2  - a \overline{\varphi}_1 \overline{e}_1 - b \overline{\varphi}_2 \overline{e}_2 \label{eq:e3}
 \end{eqnarray}
 where $\sigma_1'= \sigma_1 - a\varphi_3 + \overline{a\varphi}_3$ and  $\sigma_2'= \sigma_2 - b\varphi_3 + \overline{b\varphi}_3$.
\end{lemma}

\begin{proof}
The components of $T$ give us the expression for $\gamma$ and $\theta^2$. From $\nabla e = \theta^1e+ \overline{\theta^2} \overline{e}$ and $\theta^1=\theta^s-\gamma$, the above identities are established.
\end{proof}

Let us first analyze the case $a=0$. In this case, the distribution $E$ generated by $\{ e_1, \overline{e}_1\}$, which is globally defined, as it is contained in the eigenspace of $\phi$ with respect to the eigenvalue $0$ and orthogonal to $e_3$ and $\overline{e}_3$. From the first equation in Lemma \ref{lemma10}, we see that $E$ is parallel with respect to the Riemannian connection. Therefore if $M^3$ is complete, then its universal cover will split off a de Rham factor which is a K\"ahler curve and the other factor is a non-K\"ahler SKL surface.

Then let us concentrate on the case $ab \neq 0$. It is easy to verify that $\frac{a}{|a|}=i\frac{b}{|b|}$ in this case and we want to see a de Rham splitting into two Sasakian $3$-manifolds. For this purpose we need to identify the Reeb vector fields $\xi$ and $\xi'$. By (\ref{eq:e3}), we can form the global real vector fields with unit length as
$$ \xi = \frac{i}{\sqrt{2}|a|} ( a \overline{e}_3 -\overline{a} e_3), \ \ \ \ \xi' = \frac{i}{\sqrt{2}|b|} ( b \overline{e}_3 -\overline{b} e_3). $$
It is easy to check that $J\xi = \xi'$ and
\begin{eqnarray}
\nabla \xi & = & \sqrt{2} |a| i ( - \varphi_1 e_1 + \overline{\varphi}_1 \overline{e}_1)\\
\nabla \xi' & = & \sqrt{2} |b| i ( - \varphi_2 e_2 + \overline{\varphi}_2 \overline{e}_2)
\end{eqnarray}
In the mean time, by (\ref{eq:e1}) and (\ref{eq:e2}), it yields that
\begin{equation}\label{eq:e1e2}
\nabla e_1 = \sigma_1' e_1- i \sqrt{2}|a| \overline{\varphi}_1 \xi , \ \ \ \ \
\nabla e_2 = \sigma_2'e_2 - i \sqrt{2}|b| \overline{\varphi}_2 \xi'.
\end{equation}
Write $e_1=\frac{1}{\sqrt{2}}(Y-iZ)$ and $e_2=\frac{1}{\sqrt{2}}(Y'-iZ')$ and denote by $E$, $E'$ the distributions spanned by $\{ Y, Z, \xi\}$, $\{ Y', Z', \xi'\}$ respectively. Note that $E$ is globally defined, as $\mbox{span}\{ e_1 , \overline{e}_1\}$ is an eigenspace of $\phi$ and $e_3$ is clearly a global vector field. Similarly, $E'$ is also globally defined. The above equations says that both $E$ and $E'$ are parallel distributions with respect to the Riemannian connection $\nabla$ of $M^3$. Hence if $M^3$ is complete, they will give a de Rham decomposition on the universal cover level and each factor is a Sasakian $3$-manifold.

The $n=2$ case can be argued similarly. This completes the proof of Theorem \ref{thm8}.
\end{proof}

\begin{proof}[\textbf{Proof of Theorem \ref{thm9}}]
Let $(M^3,g)$ be a Hermitian manifold that is not balanced and its Strominger connection $\nabla^s$ has parallel torsion. We need to show that it always satisfies the LP condition in the sense of Belgun, namely, its torsion $1$-form $\eta$ obeys the equations
$$ \partial \eta = 0, \quad \partial \omega = c \eta \, \partial \overline{\eta}. $$
Fix any point $x\in M$ and let $e$ be a local unitary frame such that $\theta^s$ vanishes at $x$. Then at $x$ we have $\theta = -2\gamma$. Hence, by the equations \eqref{dvarphi} and \eqref{dbarvarphi}, it yields that $\partial \varphi = -\tau$ and $\overline{\partial}\varphi = -2 \overline{\gamma'} \,\varphi$ at $x$. The first Bianchi identity says that $d\tau = -\,^t\!\theta \tau + \,^t\!\Theta \varphi$, and taking the $(3,0)$-part at $x$, we get $\partial \tau = 2 \,^t\!\gamma' \tau $, which under the assumption $\nabla^sT=0$ leads us to the following equality
$$ \sum_r \{ T^r_{ij} T^{\ell}_{rk} + T^r_{ki} T^{\ell}_{rj} + T^r_{jk} T^{\ell}_{ri} \} = 0$$
for any indices. Take $\ell = k$ and sum up, which yields that $\sum_r \eta_r T^r_{ij}=0$ for any $i$, $j$.  Again at the point $x$, it follows that
$$ \partial \eta = \eta_r \partial \varphi_r = - \eta_r T^r_{ik} \varphi_i \varphi_k =0,$$
since $\sum_r \eta_rT^r_{ik}=0$. The first equality in the above line is due to $\nabla^sT=0$, which implies that $\nabla^s \eta =0$. Similarly, it yields that, at $x$,
$$ \partial \overline{\eta} = \sum_r \overline{\eta}_r \partial \overline{\varphi}_r = 2 \sum_{r,i,j} \overline{\eta}_r T^j_{ir} \varphi_i \overline{\varphi}_j. $$
Without loss of generality, we may assume that the frame $e$ at $x$ enjoys the property that $\eta_3\neq 0$ and $\eta_1=\eta_2=0$. Hence, at this point $x$, the equalities $T^3_{\ast \ast }=0$ and $T^{\ast }_{12}=0$ are established. The former is due to the equality $\sum_r \eta_rT^r_{ij}=0$ for any $i,j$, and the latter results from
$$ 0 = \eta_1 = T^2_{21} + T^3_{31} = T^2_{21},$$
with $T^1_{12}=0$ similarly shown. Based on this, it yields that, at $x$,
$$ \eta \partial \overline{\eta} = 2 |\eta_3|^2  \varphi_3 \sum_{i,j=1}^2  T^j_{i3} \varphi_i \overline{\varphi}_j  . $$
On the other hand, at $x$, it follows that
\[\begin{aligned}
-\sqrt{\!-\!1} \partial \omega & =\, ^t\!\tau \,\overline{\varphi} = \sum_{i,j,k} T^j_{ik} \varphi_i \varphi_k \overline{\varphi}_j \\
& =  2(T^j_{13} \varphi_1\varphi_3 + T^j_{23} \varphi_1\varphi_3)\overline{\varphi}_j \\
& = -2 \varphi_3 \sum_{i,j=1}^2 T^j_{i3} \varphi_i \overline{\varphi}_j.
\end{aligned}\]
Therefore, the LP condition is satisfied. This completes the proof of Theorem \ref{thm9}.
\end{proof}

\begin{proof}[\textbf{Proof of Theorem \ref{thm10}}]
For a non-K\"ahler SKL manifold, by \cite[Lemma 15]{ZZ}, it yields that $\eta$ satisfies
\[\partial \eta=0,\quad \partial \overline{\eta} = 2 \phi^j_i \varphi_i\overline{\varphi}_j,\]
under any unitary frame. When the admissible frame $e$ is applied, it follows that the matrix $(\phi^j_i)$ is diagonal, and by Remark \ref{admissibleT},
\[T^n_{\ast \ast}=0, \quad T^j_{in}=0\quad \text{for}\quad i \neq j.\]
It yields that
$$ \eta  \partial \overline{\eta}  = \lambda \varphi_n \wedge 2 \sum_{i=1}^{n-1} \lambda a_i \varphi_i \overline{\varphi}_i ,$$
where $\phi^i_i=\lambda a_i $ and $a_i$ is a globally defined constant. Similarly, it can be shown that
$$ -\sqrt{\!-\!1} \partial \omega = \, ^t\!\tau \, \overline{\varphi}
= \sum_{i,j,k} T^j_{ik} \varphi_i\varphi_k \overline{\varphi}_j
=\sum_{\begin{subarray}{c} i<k \\ j <n \end{subarray}}2T^j_{ik} \varphi_i\varphi_k \overline{\varphi}_j. $$
If the metric satisfies the LP condition, it follows that $\partial \omega = c \eta  \partial \overline{\eta} $ for some non-zero constant $c$, hence, the above formula implies that for any $j<n$, $T^j_{ik}=0$ unless $(i, k) = (j,n)$. In particular, $T^{\ast }_{ik}=0$ for any $i,k<n$. Conversely, if we have the degenerate torsion, it follows from Remark \ref{degenerateT} that the only possibly non-zero components of the torsion $T$ are $T^i_{in}$ for $i \leq n-1$ under any admissible frame $e$, and thus $\partial \omega$ is a non-zero constant multiple of $\eta \partial \overline{\eta}$, which is exactly the LP condition.

Let $(M^n,g)$ be a complete non-K\"ahler SKL manifold with degenerate torsion. We will show that its universal cover always splits off a K\"ahler de Rham factor, of complex codimension either $2$ or $3$. Under a strictly admissible frame $e$, the matrix $\phi$ is diagonal, which takes the special form as in Lemma \ref{lemma6}, while the connection matrix $\theta^s$ of $\nabla^s$ is block diagonal and takes the form as in Lemma \ref{lemma7}. As the notations between Lemma \ref{lemma6} and Definition \ref{sadmb}, $E \oplus \overline{E}$ is the distribution in $M$ spanned by $\{ e_1, \ldots e_{n-2};  \overline{e}_1, \ldots , \overline{e}_{n-2}\}$ in the rank $1$ case, or by $\{ e_1, \ldots e_{n-3}; \overline{e}_1, \ldots , \overline{e}_{n-3} \}$ in the rank $2$ case, with $m$ being $n-2$ in the former case and $n-3$ in the latter case. We claim that $E \oplus \overline{E}$ is parallel under the Riemannian connection $\nabla$.

By Remark \ref{sadmissibleT}, the non-zero components of the torsion are $T^i_{in}$ for $m<i<n$ under an strictly admissible frame. In particular, for any $i\leq m$, it yields that $\gamma_{ij}=0$ and $\theta^2_{ij}=0$ for any $j$. Therefore, for any $i\leq m$, it follows that
$$ \nabla e_i = \theta^1_{ij}e_j + \overline{\theta^2_{ij}} \overline{e}_j = (\theta^s_{ij} -\gamma_{ij}) e_j + \overline{\theta^2_{ij}} \overline{e}_j = \theta^s_{ij} e_j \in E, $$
since the connection matrix $\theta^s$ of $\nabla^s$ is block diagonal and takes the form as in Lemma \ref{lemma7}.
This gives us the desired de Rham splitting, and in the factor giving by $E \oplus \overline{E}$, the metric is K\"ahler, since the torsion vanishes there. Therefore we have completed the proof of Theorem \ref{thm10}.
\end{proof}

\vsv
\vsv

\noindent\textbf{Acknowledgments.} The second named author is grateful to the Mathematics Department of Ohio State University for the nice research environment and the warm hospitality during his stay. The third named author would like to thank his collaborators Qingsong Wang and Bo Yang in their previous works, which laid the foundation for the computation carried out in the present paper. We also want to thank the referees for helpful suggestions and comments which improved the readability of the manuscript.

\vs

\end{document}